\documentclass[10pt]{amsart}
\pdfoutput=1
\usepackage{amsmath}
\usepackage{amsfonts}
\usepackage{amssymb}
\usepackage{amsthm}
\usepackage{hyperref}
\usepackage{comment}
\usepackage{float}
\usepackage{youngtab}
\usepackage{ytableau}
\usepackage{rotating}
\usepackage{xcolor}
\usepackage{bbding}
\usepackage{tikz}
\usepackage{tikz-cd}
\usepackage[paper=a4paper, margin=3.5cm]{geometry}

\def\ZZ{{\NZQ Z}}
\def\RR{{\NZQ R}}

\def\PP{{\NZQ P}}

\def\frk{\frak}               

\def\Phi{{\frk n}}
\def\Phi{{\frk N}}
\def\opn#1#2{\def#1{\operatorname{#2}}} 
\opn\chara{char} \opn\length{\ell} \opn\pd{pd} \opn\rk{rk}
\opn\projdim{proj\,dim} \opn\injdim{inj\,dim} \opn\rank{rank}
\opn\spn{span}\opn\Seg{Seg}
\opn\depth{depth} \opn\grade{grade} \opn\height{height}
\opn\embdim{emb\,dim} \opn\codim{codim}

\opn\Tr{Tr} \opn\bigrank{big\,rank}
\opn\superheight{superheight}\opn\lcm{lcm}
\opn\trdeg{tr\,deg}
\opn\reg{reg} \opn\lreg{lreg} \opn\ini{in} \opn\lpd{lpd}
\opn\size{size}\opn\bigsize{bigsize}
\opn\cosize{cosize}\opn\bigcosize{bigcosize}
\opn\sdepth{sdepth}\opn\sreg{sreg}
\opn\link{link}\opn\fdepth{fdepth} \opn\trdeg{trdeg} \opn\mod{mod}
\opn\spann{span}
\opn\div{div} \opn\Div{Div} \opn\cl{cl} \opn\Cl{Cl}
\opn\Spec{Spec} \opn\Supp{Supp} \opn\supp{supp} \opn\Sing{Sing}
\opn\Ass{Ass} \opn\Min{Min}\opn\Mon{Mon} \opn\dstab{dstab} \opn\astab{astab}
\opn\Syz{Syz}
\opn\Ann{Ann} \opn\Rad{Rad} \opn\Soc{Soc} \opn\Aut{Aut}
\opn\Im{Im} \opn\Ker{Ker} \opn\Coker{Coker} \opn\Am{Am}
\opn\Hom{Hom} \opn\Tor{Tor} \opn\Ext{Ext} \opn\End{End}
\opn\Aut{Aut} \opn\id{id}

\opn\nat{nat}
\opn\pff{pf}
\opn\Pf{Pf} \opn\GL{GL} \opn\SL{SL} \opn\mod{mod} \opn\ord{ord}
\opn\Gin{Gin} \opn\Hilb{Hilb}\opn\sort{sort}
\opn\S{S} \opn\dim{dim} \opn\supp{supp}\opn\trdeg{trdeg}\opn\sort{sort}
\opn\aff{aff} \opn\con{conv} \opn\relint{relint} \opn\st{st}
\opn\lk{lk} \opn\cn{cn} \opn\core{core} \opn\vol{vol}
\opn\link{link} \opn\star{star}\opn\lex{lex}
\opn\conv{conv} \opn\Ehr{Ehr}\opn\Pic{Pic}
\opn\Conv{Conv}
\opn\gr{gr}

\def\pot#1#2{#1[\kern-0.28ex[#2]\kern-0.28ex]}

\opn\dirlim{\underrightarrow{\lim}}
\opn\inivlim{\underleftarrow{\lim}}
\def\Implies{\ifmmode\Longrightarrow \else
        \unskip${}\Longrightarrow{}$\ignorespaces\fi}
\def\implies{\ifmmode\Rightarrow \else
        \unskip${}\Rightarrow{}$\ignorespaces\fi}
\def\iff{\ifmmode\Longleftrightarrow \else
        \unskip${}\Longleftrightarrow{}$\ignorespaces\fi}

\let\:=\colon

\newtheorem{Theorem}{Theorem}[section]
 \newtheorem{Lemma}[Theorem]{Lemma}
 \newtheorem{Corollary}[Theorem]{Corollary}
 
 \newtheorem{Remark}[Theorem]{Remark}

 \newtheorem{Definition}[Theorem]{Definition}

\def\RR{\mathbb{R}}
\def\ZZ{\mathbb{Z}}
\def\PP{\mathbb{P}}

\def\fC{\mathcal{C}}

\def\zz{\ZZ_2\times \ZZ_2}

\begin{document}
 \title {Phylogenetic degrees for claw trees}
\keywords {group-based model, algebraic degree of a variety, polytope, volume}
 
 \author{Rodica Andreea Dinu}
\address{%
	University of Konstanz, Fachbereich Mathematik und Statistik, Fach D 197 D-78457 Konstanz, Germany, and Simion Stoilow Institute of Mathematics of the Romanian Academy, Calea Grivitei 21, 010702, Bucharest, Romania}
	\email{rodica.dinu@uni-konstanz.de}

\author{Martin Vodi\v{c}ka}
\address{Šafárik University, Faculty of Science, Jesenná 5, 04154 Košice, Slovakia}
	\email{martin.vodicka@upjs.sk}

\maketitle

 \begin{abstract}
Group-based models appear in algebraic statistics as mathematical models coming from evolutionary biology, respectively the study of mutations of organisms.  Both theoretically and in terms of applications, we are interested in determining the algebraic degrees of the phylogenetic varieties coming from these models. These algebraic degrees are called {\em phylogenetic degrees}. In this paper, we compute the phylogenetic degree of the variety $X_{G, n}$ with $G\in\{\ZZ_2,\zz, \ZZ_3\}$ and any $n$-claw tree. As these varieties are toric, computing their phylogenetic degree relies on computing the volume of their associated polytopes $P_{G,n}$. We apply combinatorial methods and we give concrete formulas for them.
 \end{abstract}

 \maketitle

\section{Introduction}

Evolution is the process by which modern organisms have descended from ancient ancestors \cite{fels}. Evolutionary change ultimately relies on the mutations of organisms.
Phylogenetic trees have become objects of interest for study, given the desire to understand as much as possible about evolutionary biology. They have also become of interest to mathematicians, as relationships have been found between these objects - in biology - and algebraic varieties - in algebraic geometry. 
These varieties are represented by a phylogenetic tree and an evolutionary model. In such a tree, each edge corresponds to a mutation. We consider the probabilities of the different mutations as entries of a matrix, which we call the transition matrix. Phylogeneticists distinguish certain types of matrices specific to an evolutionary model. By fixing a particular type of transition matrix, we obtain a probability distribution on the set of states of the species of interest. We get an algebraic map by fixing a model and then varying the entries of the transition matrices to obtain different probability distributions. The closure of the image of this map is a variety, that we call a {\em phylogenetic variety}. We will assume that an abelian group $G$ acts transitively and freely on the set of states. A general group-based model is a maximal subspace of transition matrices invariant under the group action. A subspace of this space is called a {\em group-based model}. For more details regarding group-based models and their geometry, the reader may consult: \cite{mateusz, mateusztor, martinko, RM}. It is known that the varieties coming from group-based models are toric, so they contain algebraic torus as a dense subset \cite{evans, sz}. There is no classification of the normality of these varieties: it is known only that when $G\in \{\ZZ_2, \zz, \ZZ_3\}$ the corresponding phylogenetic varieties to any tree are normal, see \cite{martinko, RM}, and when $|G|=2k$, $k\leq 3$ the corresponding phylogenetic variety for any tree is not normal, see \cite[Proposition~2.1]{RM}.
An important fact while working with phylogenetic varieties coming from group-based models is that the defining ideals of these varieties associated to any phylogenetic tree can be seen as toric fiber products of the defining ideals of the phylogenetic varieties associated to claw trees (i.e. trees that have only one node and $n$ leaves), \cite{seth}. This fact shows that, in some cases, one can reduce checking a property for the phylogenetic variety associated to any phylogenetic tree to checking that property only in the case when the tree is a claw tree, see \cite{draisma1} and \cite[Lemma~5.1]{mateusz}. As the Gorenstein property behaves well with respect to toric fiber products, a classification of Gorenstein Fano phylogenetic varieties coming from any $G\in \{\ZZ_2, \zz, \ZZ_3\}$ and (trivalent) tree was obtained in \cite[Theorem 5.1]{RM}. The group $\zz$ is one of the greatest biological significance, as it corresponds to the action given by the Watson-Crick complementary. Its corresponding group-based model is also called {\em the 3-Kimura parameter model}, see \cite{kimura}.

In this paper, we are interested in investigating the algebraic degrees of the phylogenetic varieties coming from group-based models. We call them {\em phylogenetic degrees}. In the literature, the phylogenetic degrees are known the phylogenetic degree only for some specific small computational examples, see \cite{examples}. As already mentioned, these varieties are toric, hence computing the phylogenetic degrees can be reduced to computing the volume of the associated polytopes, see \cite[Section~5.3]{fulton}. For a comprehensive monograph on triangulations and computing volumes of polytopes, see \cite{triangbook}.

We present the structure of this paper. In Section~\ref{prel}, we give an introduction to constructing varieties coming from phylogenetic trees and their associated polytopes. Then, each section is dedicated to the study of the phylogenetic degree of the variety associated to any claw tree and one of the groups: $G\in \{\ZZ_2, \zz, \ZZ_3\}$ as follows.\\

In Section~\ref{1}, the main result is Theorem~\ref{phylodegZ2}, from which we obtain a formula for computing the phylogenetic degree of the variety associated to the group $\ZZ_2$ and any $n$-claw tree:\\

{\bf Theorem.} The phylogenetic degree of the projective algebraic variety $X_{\ZZ_2,n}$ is $$\frac{n!}{2}-2^{n-2}.$$

In Section~\ref{2}, the main result is Theorem~\ref{phylodegZ2xZ2} from which we give a formula for computing the phylogenetic degree of the variety associated to the group $\zz$ and any $n$-claw tree:\\

{\bf Theorem.} The phylogenetic degree of the projective algebraic variety $X_{\zz,n}$ is equal to $$\frac{(3n)!}{4\cdot 6^n}-3\cdot 2^{n-3}\cdot \sum_{i=0}^n (-2)^i \binom ni \frac{(3n)!}{(2n+i)!} + 3\cdot 4^{n-2}\binom{2n}{n}-n\cdot 4^{n-1}.$$

In Section~\ref{3}, the main result is Theorem~\ref{phylodegZ3} from which we obtain a formula for computing the phylogenetic degree of the variety associated to the group $\ZZ_3$ and any $n$-claw tree:\\

{\bf Theorem.} The phylogenetic degree of the projective algebraic variety $X_{\ZZ_3,n}$ is equal to $$\frac{(2n)!}{3\cdot 2^n}-2^{n+1}\cdot 3^{n-2}+3^{n-1}\cdot n.$$    

In all three cases, we apply the same combinatorial strategy: we regard the corresponding polytopes as embedded in higher dimensional cubes and we cut off parts that do not lie in our polytopes. As these parts often intersect each other, we apply the principle of inclusion and exclusion in order to determine their volume, and moreover, the desired phylogenetic degrees.

\section{Preliminaries}\label{prel}

In this section, we introduce the notation and preliminaries used in the article. For more details, the reader may consult \cite{ss, h-rep, RM} and the references therein.
\subsection{Phylogenetic trees}\label{phylotrees}
We introduce the algebraic variety associated to a phylogenetic tree and a group, and its defining polytope. From an algebraic-geometrically point of view, the roots of these constructions may be found in \cite{erss}. More details can be consulted in \cite{bw}.

 Let $T$ be a rooted tree with a set of vertices $V=V(T)$ and a set of edges $E=E(T)$. We direct the edges of $T$ away from the root. A vertex $v$ is called a {\it leaf} if its valency is 1, otherwise it is called a {\it node}.
 The set of all leaves is denoted by $L$ and the set of all nodes is denoted by $N$. Let $S=\{\alpha_0, \alpha_1, \dots\}$ be a finite set, that is usually called the set of states, and let $W$ be a finite-dimensional vector space spanned freely by elements of $S$. Let $\widehat{W}$ be a subspace of $W\otimes W$. Any element of $\widehat{W}$ can be represented as a matrix in the basis corresponding to $S$ and this matrix is called a {\it transition matrix}. We associate to any vertex $v\in V$ a complex vector space $W_v \simeq W$, and we associate to any edge $e\in E$ a subspace $\widehat{W}_e \subset W_{v_1(e)}\otimes W_{v_2(e)}$, $\widehat{W}_e \simeq \widehat{W}$, with $e$ is directed edge from the vertex $v_1(e)$ to $v_2(e)$.
We call the tuple $(T, W, \widehat{W})$ a \textit{phylogenetic tree}.

 \subsection{A variety associated to a group-based model.}
We define the following spaces:
\[
W_V=\bigotimes_{v\in V} W_v,\;\;  W_L=\bigotimes_{l\in L} W_l,\;\;  \widehat{W}_E= \bigotimes_{e\in E} \widehat{W}_e.
\]
We call $W_V$ the space of all possible states of the tree, $W_L$ the space of states of leaves, and $\widehat{W}_E$ the parameter space.

\begin{Definition} \cite[Construction 1.5]{bw}
We consider a linear map $\widehat{\psi}: \widehat{W}_E \rightarrow W_V$ whose dual is defined as
\[
\widehat{\psi}^{*}(\bigotimes_{v\in V} \alpha_v^{*})= \bigotimes_{e\in E}(\alpha_{v_1(e)}\otimes \alpha_{v_2(e)})^{*}_{| \widehat{W}_e},
\]
where $\alpha_v$ is an element of $S$.
\end{Definition}
The map $\widehat{\psi}$ associates to a given choice of matrices the probability distribution on the set of all possible states of vertices of the tree.\\
Its associated multi-linear map $\widetilde{\psi}: \prod_{e\in E}\widehat{W}_e \rightarrow W_V$ induces a rational map of projective varieties
\[
\psi: \prod_{e\in E}(\PP(\widehat{W}_e))\dashrightarrow \PP(W_V).
\]
We consider the map $\pi: W_V \rightarrow W_L$ which is defined as
\[
\pi=(\otimes_{v\in L}\id_{W_v})\otimes (\otimes_{v\in N} \sigma_{W_v}),
\]
with $\sigma=\sum \alpha_i^{*} \in W^{*}$ which sums up all the coordinates. The map $\pi$ sums up the probabilities of the states of vertices which are different only in nodes.

\begin{Definition}\label{variety}
The image of the composition $\pi \circ \psi$ is an affine subvariety in $W_L$ and it is called the affine variety of a phylogenetic tree $(T, W, \widehat{W})$. By $X_T$ we denote its underlying projective variety in $\PP(W_L)$.
\end{Definition}

We will assume that an abelian group $G$ acts transitively and freely on the set $S$. We define $\widehat{W}$ to be the set of fixed points of the $G$ action on $W\otimes W$. In this case, the space $\mathcal{M}$ of transition matrices is given as a subspace invariant under the action of the group $G$. This choice of $\widehat{W}$ leads to the so-called {\it group-based models}. For more details about these models, the reader may consult \cite{ss, erss} and \cite{mateusztor}. We point out that the variety in Definition~\ref{variety} depends on the choice of the abelian group $G$ and for this reason, we will denote it as $X_{G, T}$ and call it the \textit{group-based phylogenetic variety}.

\subsection{The defining polytope.}
As the group-based phylogenetic variety $X_{G,T}$ is a projective toric variety, by \cite{ss}, instead of looking at the variety $X_{G,T}$, we can study the associated polytope associated to this projective toric variety. We denote it by $P_{G,T}$. When the tree, $T$, is the $n$-claw tree, we denote the corresponding polytope by $P_{G,n}$.


We introduce some notation that will be of use in the next sections, when working with polytopes $P_{G,n}\subset\mathbb R^{n|G|}$. We label the coordinates of a point $x\in\mathbb R^{n|G|}$ by $x_g^j$, where $1\le j\le n$ corresponds to an edge of the tree, and $g\in G$ corresponds to a group element. For any point $x\in \mathbb Z_{\ge 0}^{n|G|}$ we define its $G$-presentation as an $n$-tuple $(G_1,\dots,G_n)$ of multisets of elements of $G$. Every element $g\in G$ appears exactly $x_g^j$ times in the multiset $G_j$. We denote by $x(G_1,\dots,G_n)$ the point with the corresponding $G$-presentation. In the case where all multisets contain exactly one element we may simply say that the $G$-presentation is just an $n$-tuple of elements of $G$.

The vertex description of the polytopes $P_{G,n}$ is known and may be consulted in great detail in \cite{bw, mateusz, ss}.
We recall this description and we formulate it in the language of $G$-presentations.

\begin{Theorem} 
The vertices of the polytope $P_{G, n}$ associated to the finite abelian group $G$ and the $n$-claw tree are exactly the points $x(G_1,\dots,G_n)$ with $G_1+\dots+G_n=0$.

Let $L_{G, n}$ be the lattice generated by vertices of $P_{G,n}$. Then
 $$L_{G,n}=\{x\in \mathbb Z^{n|G|}:\sum_{g,j}x_g^j\cdot g=0,
\forall \text{ } 1\le j,j' \le n, \sum_g x_g^j=\sum_g x_g^{j'} \}.$$
where the first sum is taken in the group $G$.
\end{Theorem}

The polytope $P_{G,n}$ has a lot of symmetries that can be described by group actions on $\RR^{|G|n}$. All of these actions are isomorphisms of $\RR^{|G|n}$ and preserve $P_{G,n}$:
 \begin{itemize}
\item action of $G^n$: For $h\in G^n$ and $x\in \mathbb{R}^{(|G|)n}$ we define $(hx)_i^j=x_{i+h_j}^j$, where the sum $i+h_j$ is in $G$. 

This action preserve polytope $P_{G,n}$ only if we act with $h\in H_n$, where $H_{n}=\{(h_1, h_2,\dots, h_n)\in G^n: h_1+h_2+\dots+h_n=0\}$.
    \item action of $\mathbb S_n$: For $\sigma\in \mathbb S_{n}$ and $x\in \mathbb{R}^{(|G|)n}$, we define $(\sigma x)_i^{\sigma(j)}=x_{i}^j$.
    \item action of $\Aut(G)$: For $\varphi\in \Aut(G)$ and $x\in \mathbb{R}^{(|G|)n}$ we define $(\varphi x)_{\varphi(i)}^j=x_{i}^j$.
\end{itemize}

Clearly, the polytope $P_{G, n}$ in $\RR^{|G|n}$ is not full-dimensional. For this reason, for the rest of the article, we will with the projection of $P_{G,n}$ on the $(|G|-1)n$ coordinates which correspond to non-zero elements of $G$. We still denote this projection by $P_{G,n}$. Note that the above-described group actions still give us isomorphisms of $R^{(|G|-1)n}$.

We denote by $[n]:=\{1,2, \dots, n\}$. Also, for a set $A$, we denote by $\mathcal{P}_{odd}(A)$ the set of all subsets of $A$ having odd cardinality.

\vspace{.2in}
\section{Phylogenetic degrees of \texorpdfstring{$X_{\ZZ_2,n}$}{Z2}}\label{1}
\vspace{.2in}

We start by giving an informal idea of our approach: One can obtain the polytope $P_{\ZZ_2,n}$ from the cube $\fC_n=[0,1]^n$ by cutting its corners. More specifically, we need to cut off all its vertices with an odd sum of coordinates. At every such vertex, we cut off the unit simplex with lattice volume one. Since we have $2^{n-1}$ such vertices, the polytope has volume $n!-2^{n-1}$. One can imagine it nicely in 3 dimensions: We can visualize the polytope $P_{\ZZ_2,3}$ as it is a simplex by cutting off four unit simplices from the cube, see Figure~\ref{polytope}.
\begin{figure}[h!]
  \includegraphics[width=\linewidth, scale=0.9, width=14cm]{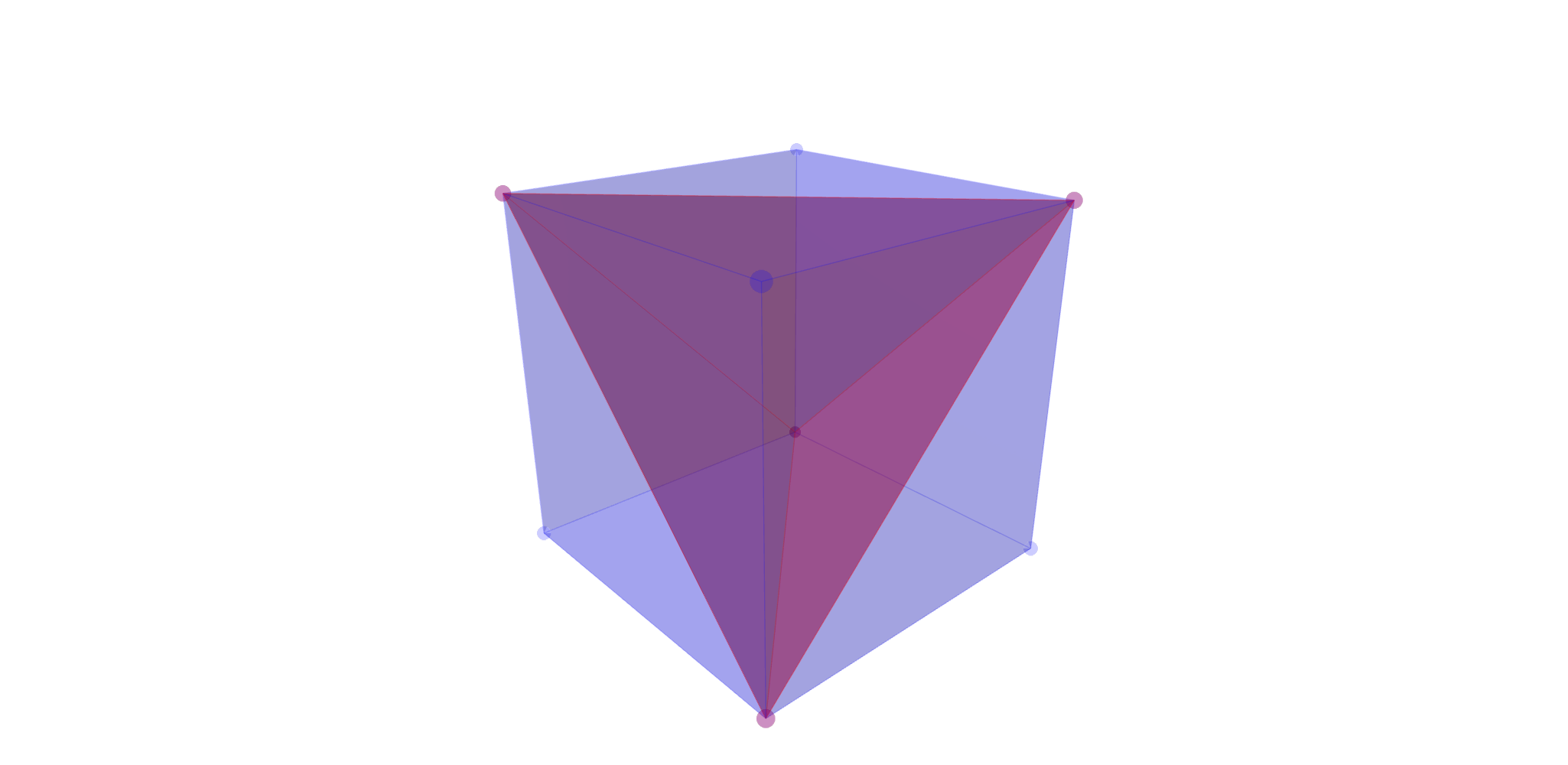}
  \caption{}
  \label{polytope}
\end{figure}

 Let us denote for $A\subset [n]$ $$H^+_A=\{x\in \RR^n: \sum_{i\not\in A} x_1^i - \sum_{i\in A} x_1^i\ge 1-|A|\},$$
$$H^-_A=\{x\in \RR^n: \sum_{i\not\in A} x_1^i - \sum_{i\in A} x_1^i\le 1-|A|\},$$
$$H_A=\{x\in \RR^n: \sum_{i\not\in A} x_1^i - \sum_{i\in A} x_1^i= 1-|A|\}.$$

\begin{Theorem}(\cite{bw, h-rep})\label{facetdescription_Z2}
The facet description of the polytope $P_{\ZZ_2,n}$ is:
\begin{itemize}
\item $x_1^i\geq 0$ for all $1\leq j\leq n$,
\item $x_{1}^j\le 1$ for all $1\leq j\leq n$,
\item For all $A\in\mathcal{P}_{odd}([n])$:
\[
\sum_{i\not\in A} x_1^i - \sum_{i\in A} x_1^i\ge 1-|A|.\]

\end{itemize}
\end{Theorem}

We will need the following results:

\begin{Lemma}\label{Z2-simplex} For all $A\subset [n]$, $\fC_n\cap H^-_A$ is a lattice polytope. Moreover, it is a unit simplex, thus its (lattice) volume is one.
\end{Lemma}
\begin{proof} Consider $g^A\in \ZZ_2^n$ such that $g^A_i=1$ for $i\in A$ and $g^A_i=0$ for $i\not\in A$. If we act with $g^A$, then $g^A(\fC_n\cap H^-_A)=\fC_n\cap H^-_\emptyset$. Thus $g_A(\fC_n\cap H^-_A)$ is a polytope defined by inequalities $0\le x_1^i\le 1$ and $x_1+\dots +x_n\le 1$ which is a unit simplex. This implies that also $\fC_n\cap H^-_A$ is a unit simplex which proves the lemma.
\end{proof}

\begin{Lemma}\label{Z2-emptyintersection}
For all different sets $A,\ B\subset [n]$ with $|A|+|B|$ being even number, the polytope $\fC_n\cap H_A^-\cap H_B^-$ is not full-dimensional, thus its (lattice) volume is zero. 
\end{Lemma}
\begin{proof} Consider $x\in \fC_n\cap H_A^-\cap H_B^-$. By summing up the inequalities for $H_A^-$ and $H_B^-$ and using $0\le x_1^i\le 1$ we obtain:

$$2-|A|-|B|\ge 2\sum_{i\not\in A\cup B} x_1^i -2\sum_{i\in A\cap B}x_1^i \ge -2|A\cap B| \Leftrightarrow$$
$$\Leftrightarrow 2\le |A|+|B|-2|A\cap B|=|A\triangle B|,$$

where $A\triangle B$ is the symmetric difference of $A$ and $B$. Since $A$ and $B$ are different sets and $|A|+|B|$ is even we have $|A\triangle B|\ge 2$. Therefore there must be equality in all inequalities above. In particular, there is an equality in the inequality for $H_A^-$, thus $x_1^i\in H_A$. That means the whole polytope $\fC_n\cap H_A^-\cap H_B^-$ lies in the hyperplane $H_A$ and is not full-dimensional.  
\end{proof}
\begin{Remark}
The set $\fC_n\cap H_A^-\cap H_B^-$ must not be a lattice polytope. Since it is an intersection of halfspaces it must be either a polytope (with not necessarily integral vertices) or an empty set. In fact, one can prove that it is always a lattice polytope or an empty set, but for the purpose of this article, it is sufficient to know that its volume is zero.
\end{Remark}

Now we prove the theorem.

\begin{Theorem}\label{Z2-volume}
For all $n\ge 2$ the lattice volume of the polytope $P_{\ZZ_2,n}$ in the lattice $L_{\ZZ_2,n}$ is $n!/2-2^{n-2}$.
\end{Theorem}

\begin{proof}
From the facet description of $P_{Z_2,n}$, Theorem~\ref{facetdescription_Z2}, we know that $$P_{\ZZ_2,n}=\fC_n\cap \bigcap_{A\in \mathcal{P}_{odd}([n])} H^+_A.$$

Thus, $$P_{\ZZ_2,n}= \fC_n \setminus \bigcup_{A\in\mathcal{P}_{odd}([n])} (\fC_n\cap H^-_A).$$

By Lemma \ref{Z2-emptyintersection} we know that the volume of $(\fC_n\cap H^-_A)\cap (\fC_n\cap H^-_B)$ for different $A,B$ is 0. Thus
$$V_{\ZZ^n}(P_{\ZZ_2,n})=V_{\ZZ^n}(\fC_n)-\sum_{A\in\mathcal{P}_{odd}([n])} V_{\ZZ^n} (\fC_n\cap H^-_A)=n!- 2^{n-1}\cdot 1. $$

We are interested in the lattice volume in the lattice $L_{\ZZ_2,n}$ which is a sublattice of $\ZZ^n$ of index 2. Therefore $$V_{L_{\ZZ_2,n}}(P_{\ZZ_2,n})=V_{\ZZ^n}(P_{\ZZ_2,n})/2=n!/2-2^{n-2}.$$ 

\end{proof}

\begin{Corollary}\label{phylodegZ2}
The phylogenetic degree of the projective algebraic variety $X_{\ZZ_2,n}$ is $n!/2-2^{n-2}$.
\end{Corollary}

\vspace{.2in}
\section{Phylogenetic degrees of \texorpdfstring{$X_{\ZZ_2\times \ZZ_2, n}$}{Z2xZ2}}\label{2}
\vspace{.2in}

In this section, we will compute the volume of polytopes $P_{\zz,n}$. Our approach is very similar to the case of the group $\ZZ_2$. We start with the product of simplices $C_{3,n}:=(\Delta_3)^n$, where $\Delta_3=\{x\in\RR^3: x_1,x_2,x_3\ge 0,x_1+x_2+x_3\le 1\}$ is the 3-dimensional unit simplex.
Then we cut off the parts which are not in the polytope $P_{\zz,n}$. However, now it would be more complicated because the parts which we cut off do intersect, so we will need to use the well-known principle of inclusion and exclusion and also compute volumes of the intersection of the parts which we are cutting off.
We will need the inequalities which describe the polytopes $P_{\zz,n}$. Let us denote elements of $\zz=\{0, \alpha, \beta, \gamma\}$. Let us denote for any set $A\subset[n]$:
$$S_{A,\alpha}(x)=\sum_{i\not\in A} (x_\beta^i+x_\gamma^i) - \sum_{i\in A} (x_\beta^i+x_\gamma^i),$$
$$S_{A,\beta}(x)=\sum_{i\not\in A} (x_\alpha^i+x_\gamma^i) - \sum_{i\in A} (x_\alpha^i+x_\gamma^i),$$
$$S_{A,\gamma}(x)=\sum_{i\not\in A} (x_\alpha^i+x_\beta^i) - \sum_{i\in A} (x_\alpha^i+x_\beta^i).$$

\begin{Theorem}(\cite{h-rep})\label{facetdescription_zz}
The facet description of the polytope $P_{\zz,n}$ is:
\begin{itemize}
\item $x_g^j\geq 0$ for all $g\in \{0, \alpha, \beta, \gamma\}$, $1\leq j\leq n$,
\item $x_{\alpha}^j+x_{\beta}^j+x_{\gamma}^j\le 1$ for all $1\leq j\leq n$,
\item For all $A\in\mathcal{P}_{odd}([n])$
\[
S_{A,\alpha}\geq 1-|A|,\ S_{A,\beta}\geq 1-|A|,\ S_{A,\gamma}\geq 1-|A|.\]

\end{itemize}
\end{Theorem}

Similarly, as in the $\ZZ_2$ case, let us denote for $A\subset [n]$ and $g\in\{\alpha,\beta,\gamma\}$: 
$$H^+_{A,g}=\{x\in \RR^{3n}: S_{A,g}(x)\ge 1-|A|\},$$
$$H^-_{A,g}=\{x\in \RR^{3n}: S_{A,g}(x)\le 1-|A|\},$$
$$H_{A,g}=\{x\in \RR^{3n}: S_{A,g}(x)= 1-|A|\}.$$

Notice that $$P_{\zz,n}=\fC_{3,n}\cap \left(\bigcap_{g\in\{\alpha,\beta,\gamma\}}\bigcap_{A\in \mathcal P_{\text{odd}}([n])} (\fC_{3,n}\cap H_{A,g}^+) \right).$$

\begin{Lemma}\label{z22-emptyintersection}
For all $g\in\{\alpha,\beta,\gamma\}$ and all different sets $A,\ B\subset [n]$ with $|A|+|B|$ being even number, the polytope $C_{3,n}\cap H_{A,g}^-\cap H_{B,g}^-$ is not full-dimensional, thus its (lattice) volume is zero.
\end{Lemma}

\begin{proof}
Without loss of generality, we may consider the case $g=\alpha$. Consider the surjective linear map $\pi_n:\RR^{3n}\rightarrow \RR^n$ defined by $(\pi_n)(x)_i=x_\beta^i+x_\gamma^i$. Notice that $\pi_n(C_{3,n})=\fC_n$ and $\pi_n(H_{A,\alpha}^-)=H_A^-$. Therefore $\pi_n(C_{3,n}\cap H_{A,\alpha}^-\cap H_{B,\alpha}^-)=\pi_n(\fC_n\cap H_{A}^-\cap H_B^-)$ which is not full-dimensional by Lemma~\ref{Z2-emptyintersection}. It follows that also $\fC_{3,n}\cap H_{A,\alpha}^-\cap H_{B,\alpha}^-$ is not full-dimensional and has volume 0.
\end{proof}

\begin{Lemma}\label{z22-integers}
Let $A\subset[n]$ and let $c_g^i, d_g^i,c\in \ZZ$ for $i\in[n], g\in \zz$ and $h\in\{\alpha,\beta,\gamma\}.$ Let $P\subset \RR^{3n}$ be the polytope defined by the inequalities:
\begin{itemize}
\item $c_g^i\le x_g^i\le d_g^i$, for $i\in[n], g\in {\alpha,\beta,\gamma} $ 
\item $c_0^i-1\ge x_\alpha^i+x_\beta^i+x_\gamma^i\ge d_0^i-1$, for $i\in [n]$
\item $S_{A,h}\ge c$. 
\end{itemize}
Then $P$ is a lattice polytope, i.e. all of its vertices are in $\ZZ^{3n}$.
\end{Lemma}
\begin{proof}
 We can assume, without loss of generality, that $h=\alpha$. Let $v$ be a vertex of $P$. Then the point $v$ must give equality in at least $3n$ inequalities defining $P$. More specifically, $v$ is the solution to a system of $3n$ linearly independent equations arising from the inequalities for $P$. We will show that the solution of such a system of equations is always integral.

Obviously, we can not use both $x_g^i=c_g^i,x_g^i=d_g^i$ (unless $c_g^i=d_g^i$, but in this case, it is again just one equation). Since there is no restriction on the numbers $c_g^i,d_g^i$ except that they are integers, we can assume we may always pick the equality $x_g^i=c_g^i$. We can do that because we are no longer using the inequalities for $P$, we are only proving that some system of equations has an integral solution.

Thus, we have to choose $3n$ independent equation from the following equations:
\begin{itemize}
\item $(a)_g^i$:\;\;\; $x_g^i=c_g^i$, for all $i\in[n],g\in\{\alpha,\beta,\gamma\}$,     
\item $(a)_0^i$:\;\;\; $c_0^i-1= x_\alpha^i+x_\beta^i+x_\gamma^i$, for all $i\in[n]$,
\item $(b)$:\;\;\; $S_{A,\alpha}=c$.
\end{itemize}

For any $i\in[n]$ we can not use equations $(a)_g^i$ for all $g\in\zz$, because they are linearly dependent. We can assume that we do not take any equation of the type $(a)_0^i$. Indeed, for every $i$ let $g_i\in\zz$ be such an element for which we do not use the equation $(a)_g^i$. Now if we act with $g=(g_1,\dots,g_n)\in(\zz)^n$ on the point $v$, we get to the situation where we do not use any equation of type $(a)_0^i$. This can be checked immediately: it is possible that we get different constants, a different set $A$, but we have no restriction on them anyway.

If we do not take the equation $(b)$, the solution to the other $3n$ equations is $x_g^i=c_g^i$ which is integral. If we take the equation $(b)$ and $(3n-1)$ equation of the form $(a)_g^i$, again it can be easily seen after plugging in $(3n-1)$ values to the equation $(b)$ that also, in this case, the solution is integral. 
\end{proof}
\begin{Remark}
The statement of the lemma is true even if $P$ is empty since then it has no vertices. It is also true if we omit some inequalities since we can always add them back with sufficiently large or small constants.
\end{Remark}

\begin{Lemma}\label{l=2}
Let $P_1\subset \RR^{n_1}, P_2\subset \RR^{n_2}$ be two full-dimensional lattice polytopes with vertex 0. Let $Q\subset \RR^{n_1+n_2}$ be the polytope $\conv(P_1\times\{0\} \cup \{0\}\times P_2)$. Then $$V_{\ZZ^{n_1+n_2}}(Q)=V_{\ZZ^{n_1}}(P_1)\cdot V_{\ZZ^{n_2}}(P_2).$$ 
\end{Lemma}
\begin{proof}
First, let us consider the case where $P_1$ and $P_2$ are simplices. Then $Q$ is also a simplex. The lattice volume of $P_1$ is the determinant of matrix $M_1$ whose rows are vertices of $P_1$. Similarly, the lattice volume of $P_2$ is the determinant of matrix $M_2$. Then the lattice volume of $Q$ is the determinant of the block-diagonal matrix $M$ with blocks $M_1$ and $M_2$. Clearly, $\det( M)=\det(M_1)\cdot \det (M_2)$.

In the general case, we consider the triangulations $T_1=\{\sigma_1^i\}_{1\le i \le k_1}, T_2=\{\sigma_2^i\}_{1\le i \le k_2}$ of $P_1, P_2$ such that all simplices in $T_1, T_2$ have 0 as a vertex. Such triangulation always exists: First one triangulates (in any way) the facets of $P_1, P_2$ not containing 0, and then one just adds vertex 0 to all of these simplices. 

We claim that $$T=\{\sigma^{i,j}:=\conv (\sigma_1^i\times\{0\} \cup \{0\}\times \sigma_2^j)\}_{1\le i \le k_1,1\le j \le k_2}$$ is a triangulation of $Q$. Indeed, consider any point $(x,y)\in Q$, $x\in \RR^{n_1},y\in \RR^{n_2}$. 
Then $(x,y)=a(x',0)+(1-a)(0,y')$, where $x'\in P_1, y'\in P_2$. The points $(x',0), (0,y')$ belong to some simplices $\sigma_1^i$ and $\sigma_2^j$ and, therefore, $(x,y)\in \sigma^{i,j}$. This proves that simplices in $T$ cover $Q$.

Now we consider two different simplices $\sigma^{i,j}, \sigma^{i',j'}$ in $T$ where, without loss of generality, $i\neq i'$. We project to first $n_1$ coordinates and we obtain $\sigma_1^i$ and $\sigma_1^{i'}$ which intersection has volume 0 since $T_1$ is a triangulation. Thus also the volume of $\sigma^{i,j}\cap \sigma^{i',j'}$ is 0, which shows that $T$ is the triangulation. 

It follows that $$V_{\ZZ^{n_1+n_2}}(Q)=\sum_{i=1}^{k_1}\sum_{j=1}^{k_2} V_{\ZZ^{n_1+n_2}}(\sigma^{i,j})=\sum_{i=1}^{k_1}\sum_{j=1}^{k_2}V_{\ZZ^{n_1}}(\sigma_1^i)\cdot V_{\ZZ^{n_2}} (\sigma_2^j)= $$
$$=\sum_{i=1}^{k_1}V_{\ZZ^{n_1}}(\sigma_1^i)\cdot\sum_{j=1}^{k_2}V_{\ZZ^{n_2}} (\sigma_2^j)=V_{\ZZ^{n_1}}(P_1)\cdot V_{\ZZ^{n_2}}(P_2). $$

\end{proof}

\begin{Corollary}\label{product}
Let $P_1\subset \RR^{n_1}, P_2\subset \RR^{n_2}, \dots, P_l\subset \RR^{n_l}$ be full-dimensional lattice polytopes with vertex 0. Let $Q\subset \RR^{n_1+n_2+\dots+n_l}$ be the polytope
$$\conv\left(\bigcup_{i=1}^l\{0\}^{n_1+\dots+n_{i-1}}\times P_i\times\{0\}^{n_{i+1}+\dots+n_l}\right).$$ 
Then $$V_{\ZZ^{n_1+n_2+\dots+n_l}}(Q)=V_{\ZZ^{n_1}}(P_1)\cdot V_{\ZZ^{n_2}}(P_2)\cdot ... \cdot V_{\ZZ^{n_l}}(P_l).$$ 
\end{Corollary}
\begin{proof}
One can proceed by induction on $l$. The case $l=2$ was shown in Lemma~\ref{l=2}.
\end{proof}

\begin{Lemma}\label{z22-1facet}
For all $g\in\{\alpha,\beta,\gamma\}$ and all sets $A\subset [n]$, the set $\fC_{3,n}\cap H_{A,g}^-$ is a lattice polytope. Its lattice volume in the lattice $\ZZ^{3n}$ is $$\sum_{i=0}^n (-2)^i \binom ni \frac{(3n)!}{(2n+i)!}.$$
\end{Lemma}
\begin{proof}
For the beginning, we notice that $\fC_{3,n}\cap H_{A,g}^-$ is a lattice polytope, by Lemma \ref{z22-integers}. Without loss of generality, we may assume that $g=\alpha$. Consider $g^A\in(\zz)^n$ such that $g^A_i=\alpha$ for $i\in A$ and $g^A_i=0$ for $i\not\in A$. Then $g^A(\fC_{3,n}\cap H_{A,\alpha}^-)=\fC_{3,n}\cap H_{\emptyset,\alpha}^-$, thus, without loss of generality, we may assume also that $A=\emptyset.$

Consider any set $B\subset [n]$. Let us define a polytope $Q_B$ by the following inequalities:
\begin{itemize}
\item $x_g^i\ge 0$, for all $i\in[n]$, $g\in\{\alpha,\beta, \gamma\},$
\item $x_\alpha^i\le 1$, for all $i\in[n]$,
\item $S_{\emptyset,\alpha}(x)\le 1$,
\item $x_\alpha^i+x_\beta^i+x_\gamma^i\ge 1$, for all $i\in B$.
\end{itemize}

Note that, clearly $Q_B\subset [0,1]^{3n}$, thus, it is a polytope. Moreover, by Lemma~\ref{z22-integers}, $Q_B$ is a lattice polytope for any set $B$.

Also note that $$\fC_{3,n}\cap H_{\emptyset,\alpha}^-=Q_\emptyset\setminus \bigcup_{i=1}^n Q_{\{i\}} \text{ and } Q_B=\bigcap_{i\in B} Q_{\{i\}}.$$

Therefore, by the principle of inclusion and exclusion, we get $$(*):\;\;\;\; V_{\ZZ^{3n}}(\fC_{3,n}\cap H_{\emptyset,\alpha}^-)=\sum_{B\subset [n]} (-1)^{|B|}V_{\ZZ^{3n}}(Q_B).$$

We will finish the proof of the lemma by computing the lattice volume of the polytopes $Q_B$.

First, we note that the only inequalities concerning coordinates $x_\alpha^i$, for $i\not\in B$ are $0\le x_\alpha^i\le 1$. Therefore $Q_B=Q'_B\times [0,1]^{n-|B|}$ where $Q'_B$ is the projection of $Q_B$ on the other $2n+|B|$ coordinates.

Next, we consider the coordinates $x_\beta^j$ for $j\not\in B$. We know that all vertices of $Q'_B$ are integral and all of them have coordinate $x_\beta^j$ equal to 0 or 1. The fact that $x_\beta^j\le 1$ follows from $S_{\emptyset,\alpha}(x)\le 1$. Let $v$ be a vertex of $Q'_B$ such that $v_\beta^j=1$. From $S_{\emptyset,\alpha}(v)\le 1$ it follows that $v_g^i=0$, for all $i\in [n], g\in\{\beta,\gamma\}, (i,g)\neq (j,\beta) $. Then from $v_\alpha^i\le 1$ and $v_\alpha^i+v_\beta^i+v_\gamma^i\ge 1$ it follows that $v_\alpha^i=1$ for all $i\in B$.

We conclude that there is a unique vertex $v$ with $v_\beta^j=1$. Analogously, there is a unique vertex $v$ with $v_\gamma^j=1$. This is true for all $j\not\in B$. This implies that the lattice volume of $Q'_B$ is the same as the lattice volume of the polytope $\{x\in Q'_B;\ \forall i\not\in B: x_\beta^i=x_\gamma^i=0\}$ in the corresponding $3|B|$-dimensional lattice.

To simplify, we consider the affine transformation $\varphi$, where $\varphi(x)^i_\alpha=1-x_\alpha^i$ for all $i\in B$ and $\varphi(x)^i_g=
x_g^i$ for all other pairs $(i,g)$. We denote $Q''_B$ the projection to $3|B|$ coordinates $x_g^i$ for $i\in B$ of the polytope $\{x\in \varphi(Q'_B);\ \forall i\not\in B: x_\beta^i=x_\gamma^i=0\}$. Notice that $Q''_B$ still has the same lattice volume as $Q'_B$. To sum up, the polytope $Q''_B\subset \RR^{3|B|}$ is defined by the following inequalities:

\begin{itemize}
\item $x_g^i\ge 0$, for all $i\in B$, $g\in\{\alpha,\beta, \gamma\},$
\item $x_\alpha^i\le 1$, for all $i\in B$,
\item $\sum_{i\in A\cap B} (x_\beta^i+x_\gamma^i)-\sum_{i\in B\setminus A}  (x_\beta^i+x_\gamma^i) \le 1$,
\item $-x_\alpha^i+x_\beta^i+x_\gamma^i\ge 0$, for all $i\in B$.
\end{itemize}

Notice that the inequalities $x_\alpha^i\le 1$ are actually redundant, but this is not an issue. Since we know that all vertices of $Q''_B$ are in $\{0,1\}^{3|B|}$, computing all of them is straightforward:

$$0,\ e_\beta^i,\ e_\gamma^i,\ e_\alpha^i+e_\beta^i,\ e_\alpha^i+e_\gamma^i \text{ for all } i\in B.$$

Now we use Corollary~\ref{product} for the polytopes $R_i,i\in B$ where $R_i=\conv\{0,\ e_\beta^i,\ e_\gamma^i,\ e_\alpha^i+e_\beta^i,\ e_\alpha^i+e_\gamma^i\}$. By straightforward computation, all polytopes $R_i$ have lattice volume 2, therefore the lattice volume of $Q''_B$ is $2^{|B|}$.

It follows that the lattice volume of $Q'_B=2^{|B|}$ and $V(Q'_B)=2^{|B|}/(2n+|B|)!$. Consequently $V(Q_B)=2^{|B|}/(2n+|B|)!$ and $V_{\ZZ^{3n}}(Q_B)=2^{|B|}\cdot (3n)!/(2n+|B|)!$.

Plugging it into $(*)$ we obtain:

$$V_{\ZZ^{3n}}(\fC_{3,n}\cap H_{\emptyset,\alpha}^-)=\sum_{B\subset [n]} (-1)^{|B|}\cdot 2^{|B|}\cdot \frac{(3n)!}{(2n+|B|)!}=\sum_{i=0}^n \binom ni (-2)^{i} \frac{(3n)!}{(2n+i)!},$$

which concludes the lemma.
\end{proof}

\begin{Lemma}\label{z22-2facet}
For all $g,h\in\{\alpha,\beta,\gamma\}$, $g\neq h$ and all sets $A,B\subset [n]$, the lattice volume in the lattice $\ZZ^{3n}$ of the polytope $\fC_{3,n}\cap H_{A,g}^-\cap H_{B,h}^-$ is equal to $\binom{2n}{n}-n/2^{n-1}$.
\end{Lemma}
\begin{Remark}
From the statement, it follows directly that $\fC_{3,n}\cap H_{A,g}^-\cap H_{B,h}^-$ cannot be a lattice polytope, since its lattice volume is not an integer. We still use the convention that $V_{\ZZ^{3n}}(P)=V(P)\cdot(3n)!$.
\end{Remark}
\begin{proof}
Without loss of generality, we may assume that $g=\alpha, h=\beta.$ Consider $g^{A,B}\in(\zz)^n$ such that
\begin{itemize}
\item $g^{A,B}_i=0$ for $i\notin A\cup B$,
\item $g^{A,B}_i=\alpha$ for $i\in A\setminus B$, 
\item $g^{A,B}_i=\beta$ for $i\in B\setminus A$, 
\item $g^{A,B}_i=\gamma$ for $i\in A\cap B$.
\end{itemize}
Then $g^{A,B}(\fC_{3,n}\cap H_{A,\alpha}^-\cap H_{B,\beta}^-)=\fC_{3,n}\cap H_{\emptyset,\alpha}^-\cap H_{\emptyset,\beta}^-$. This implies that we may also assume $A=B=\emptyset$.

Similarly as in Lemma~\ref{z22-1facet}, we define polytopes $Q_j$ for $j\in [n]$ by the following inequalities:

\begin{itemize}
\item $x_g^i\ge 0$, for all $i\in[n]$, $g\in\{\alpha,\beta, \gamma\},$
\item $S_{\emptyset,\alpha}(x)\le 1$, $S_{\emptyset,\beta}(x)\le 1$,
\item $x_\alpha^j+x_\beta^j+x_\gamma^j\ge 1$,
\end{itemize}

We also define the polytope $Q_0$ by omitting the last inequality in the description of $Q_j$. 

We note that the conditions $S_{\emptyset,\alpha}(x)\le 1$, $S_{\emptyset,\beta}(x)\le 1$ implies $x_g^i\le 1$, thus $Q_j\subset [0,1]^{3n}$ and therefore it is a polytope. However, as we will see soon, $Q_j$ are not lattice polytopes.

Firstly, we consider the intersection $Q_j\cap Q_{j'}$ for $j,j'\in [n]$, $j\neq j'$. Consider $x\in Q_j\cap Q_{j'}$. We note that 

$$2\le x_\alpha^j+x_\beta^j+x_\gamma^j+x_\alpha^{j'}+x_\beta^{j'}+x_\gamma^{j'}\le S_{\emptyset,\alpha}(x)+S_{\emptyset,\beta}(x)\le 2.$$

This means equality must hold everywhere, in particular, $x_\gamma^j=0$ which means that $Q_j\cap Q_{j'}$ lies in the hyperplane given by $x_\gamma^j=0$ and, therefore, $V(Q_j\cap Q_{j'})=0$.

Next, we see that $$\fC_{3,n}\cap H_{\emptyset,\alpha}^-\cap H_{\emptyset,\beta}^-=Q_0\setminus \bigcup_{j=1}^n Q_j.$$
and, therefore,
$$(**):\;\;\;\; V_{\ZZ^{3n}}(\fC_{3,n}\cap H_{\emptyset,\alpha}^-\cap H_{\emptyset,\beta}^-)=V_{\ZZ^{3n}}(Q_0)-\sum_{j=1}^n V_{\ZZ^{3n}}(Q_j).$$

It remains to compute the lattice volume of $Q_0$ and $Q_j$ with $j\in [n]$. We start with $Q_0$.

Consider the point $e_\gamma^i$. It can be easily checked that $e_\gamma^i\in Q_0$. Moreover, the point $e_\gamma^i$ gives us equality in all inequalities except $x_\gamma^i=0$. We can pick $3n$ linearly independent out of them, from which it follows that $e_\gamma^i$ is a vertex of $Q_0$. Since it lies on all except one facet of $Q_0$, all other vertices must lie on the last facet $x_\gamma^i=0$.

From this, it follows that the lattice volume of $Q_0$ is the same as the lattice volume of $Q'_0$ which is the projection of $\{x\in Q_0: \forall i\in [n],\; x_\gamma^i=0\}$ to $2n$ coordinates $x_\alpha^i,x_\beta^i$.

The polytope $Q'_0\subset \RR^{2n}$ is defined by the following inequalities:

\begin{itemize}
\item $x_g^i\ge 0$, for all $i\in[n]$, $g\in\{\alpha,\beta\},$
\item $\sum_{i=1}^n x_\alpha^i \le 1$,\; $\sum_{i=1}^n x_\beta^i \le 1$.
\end{itemize}

From this description, it can be seen that $Q'_0$ is the product of two $n$-dimensional unit simplices. Therefore, $V(Q'_0)=(1/n!)^2$ and $V_{\ZZ^{2n}}(Q'_0)=(2n!)/((n!)^2)=\binom{2n}{n}$. Finally, the lattice volume of $Q_0$ is also $\binom{2n}{n}$.

Now we compute the lattice volumes of polytopes $Q_j$. Without loss of generality, we may assume $j=1$. Notice that the inequalities $x_\alpha^1\ge 0$ and $x_\beta^1\ge 0$ in the facet description of $Q_1$ are redundant: indeed, from $S_{\emptyset,\alpha}(x)\le 1$ and $x_g^j\ge 0$ for $j>1$ it follows that $x_\beta^1+x_\gamma^1\le 1$. Combining with $x_\alpha+x_\beta^1+x_\gamma^1\ge 1$ we obtain $x_\alpha^1 \ge 0$. Analogously for the inequality $x_\beta^1\ge 0$. Thus, the polytope $Q_1$ is defined by $3n+1$ inequalities and, therefore, it is a simplex. We list all of its vertices and one can easily check that each of them satisfies $3n$ equalities and one inequality:

$$e_\alpha^1+e_\beta^1,\ e_\beta^1,\ e_\alpha^1,\ e_\gamma^1,$$
$$e_\alpha^i+e_\beta^1,\ e_\beta^i+e_\alpha^1,\ \frac 12 (e_\gamma^i+e_\alpha^1+e_\beta^1) \text{ for } 2\le i\le n.  $$

We can compute the lattice volume of this simplex by computing the determinant of the matrix whose rows are differences of all vertices with the vertex $e_\alpha^1+e_\beta^1$. Therefore, we are computing the determinant of the matrix with rows:

$$ -e_\alpha^1,\ -e_\beta^1,\ e_\gamma^1-e_\alpha^1+e_\beta^1.$$
$$e_\alpha^i-e_\alpha^1,\ e_\beta^i-e_\beta^1,\ \frac 12 (e_\gamma^i-e_\alpha^1-e_\beta^1) \text{ for } 2\le i\le n,  $$

This is a lower triangular matrix and, therefore, its determinant is the product of the numbers on the diagonal which is $(-1)^2\cdot (1/2)^{n-1}=1/2^{n-1}.$ We conclude that the lattice volume of every $Q_j$ is $(-1)^2\cdot (1/2)^{n-1}=1/2^{n-1}$.

Now we plug the volumes of $Q_0$ and $Q_j$, with $j\in[n]$, in $(**)$ to obtain:

$$V_{\ZZ^{3n}}(\fC_{3,n}\cap H_{\emptyset,\alpha}^-\cap H_{\emptyset,\beta}^-)=\binom{2n}{n}-\sum_{j=1}^n\frac{1}{2^{n-1}}=\binom{2n}{n}-\frac{n}{2^{n-1}}.$$

\end{proof}

\begin{Lemma}\label{z22-3facet}
For all sets $A,B,C\subset [n]$ such that $|A|+|B|+|C|$ is odd, the lattice volume of the polytope $\fC_{3,n}\cap H_{A,\alpha}^-\cap H_{B,\beta}^-\cap H_{C,\gamma}^-$ in the lattice $\ZZ^{3n}$ is equal to:
\begin{itemize}
\item $4-3/2^{n-1}$,\; if $|A\setminus(B\cup C)|+|B\setminus(A\cup C)|+|C\setminus(A\cup B)|+|A\cap B\cap C|=1,$
\item 0,\; otherwise.
\end{itemize}
\end{Lemma}

\begin{proof}
We denote $$\Delta(A,B,C)=\Bigl(A\setminus(B\cup C)\Bigr)\cup \Bigl(B\setminus(A\cup C\Bigr)\cup\Bigl(C\setminus(A\cup B)\Bigr)\cup\Bigl(A\cap B\cap C\Bigr).$$
First, we notice that 
\[
|\Delta(A,B,C)|=|A|+|B|+|C|-2|A\cap B|-2|A\cap C|-2|B\cap C|+4|A\cap B\cap C|.
\]
which means $|\Delta(A,B,C)|$ is an odd number. Thus, $|\Delta(A,B,C)|\neq 1$ implies $|\Delta(A,B,C)|\ge 3$. We show that in this case $\fC_{3,n}\cap H_{A,\alpha}^-\cap H_{B,\beta}^-\cap H_{C,\gamma}^-$ is not full-dimensional and, therefore, it has (lattice) volume 0.

As in the proof of Lemma~\ref{z22-2facet}, we consider $g^{A,B}\in(\zz)^n$ such that
\begin{itemize}
\item $g^{A,B}_i=0$ for $i\notin A\cup B$,
\item $g^{A,B}_i=\alpha$ for $i\in A\setminus B$, 
\item $g^{A,B}_i=\beta$ for $i\in B\setminus A$, 
\item $g^{A,B}_i=\gamma$ for $i\in A\cap B$.
\end{itemize}

One can check that $g^{A,B}(H_{A,\alpha}^-)=H_{\emptyset,\alpha}^-$, $g^{A,B}(H_{B,\beta}^-)=H_{\emptyset,\beta}^-$, $g^{A,B}(H_{C,\gamma}^-)=H_{\Delta(A,B,C),\gamma}^-$.

Thus we may restrict to the case $A=B=\emptyset$, and $C$ is any set with odd cardinality with at least 3 elements.

Consider $x\in \fC_{3,n}\cap H_{\emptyset,\alpha}^-\cap H_{\emptyset,\beta}^-\cap H_{C,\gamma}^-$. By summing up the inequalities for $H_{\emptyset,\alpha}^-$, $H_{\emptyset,\beta}^-$ and  $H_{C,\gamma}^-$, we get $$3-|C|\ge 2\sum_{i\not\in C} (x_\alpha^i+x_\beta^i+x_\gamma^i)+\sum_{i\in C} 2x_\gamma^i\ge 0.$$  

This means that for $|C|\ge 5$ it is an empty set. For $|C|=3$ there must be equality everywhere. In particular, $x_\gamma^i=0$ for any $i\in C$, which means that $\fC_{3,n}\cap H_{\emptyset,\alpha}^-\cap H_{\emptyset,\beta}^-\cap H_{C,\gamma}^-$ lies in the hyperplane $x_\gamma^i=0$ for any $i\in C$, thus, it is not full-dimensional.

Now we move to the case $|\Delta(A, B, C)|=1$. As in the previous case, we may act with $g^{A, B}$ to get to the situation where $A=B=\emptyset$, $|C|=1$. Without loss of generality, $C=\{1\}$. To make it more symmetric, we act with $(\gamma,0,0,\dots,0)$ to get to the situation $A=B=C=\{1\}$.

We define the polytopes $Q_g$, for $g\in\zz$ by the following inequalities:

\begin{itemize}
\item $x_g^i\ge 0$, for all $2\le i\le n$, $g\in\{\alpha,\beta, \gamma\},$
\item $S_{\{1\},\alpha}(x)\le 0$, $S_{\{1\},\beta}(x)\le 0$, $S_{\{1\},\gamma}(x)\le 0$,
\item $x_\alpha^1+x_\beta^1+x_\gamma^1\le 1$,
\item $x_g^1\le 0$.
\end{itemize}

In the case $g=0$, the last inequality is omitted. 

Summing up the inequalities, we obtain 
\[0\ge S_{\{1\},\alpha}(x)+S_{\{1\},\beta}(x)\le 0+S_{\{1\},\gamma}(x)
\]
\[
=2\sum_{i=2}^n (x_\alpha^i+x_\beta^i+x_\gamma^i)-2(x_\alpha^1+x_\beta^1+x_\gamma^1)\ge 2\sum_{i=2}^n (x_\alpha^i+x_\beta^i+x_\gamma^i)-2.
\]

The last inequality implies $\sum_{i=2}^n (x_\alpha^i+x_\beta^i+x_\gamma^i)\le 1$ and $(x_\alpha^i+x_\beta^i+x_\gamma^i)\le 1$ for all $2\le i\le n$. From this, it follows that 
$$\fC_{3,n}\cap H_{\{1\},\alpha}^-\cap H_{\{1\}\beta}^-\cap H_{\{1\},\gamma}^-=Q_0\setminus \Bigl(Q_\alpha \cup Q_\beta \cup Q_\gamma \Bigr).$$

If we consider $x\in Q_\alpha\cap Q_\beta$, from $S_{\{1\},\gamma}\le 0$, we obtain that $x_\alpha^1+x_\beta^1\le 0$. Then the inequalities $x_\alpha^1,x_\beta^1\le 0$ force $x_\alpha^1=x_\beta^1=0$ which implies that $Q_\alpha\cap Q_\beta$ is not full-dimensional. Analogously $Q_\beta\cap Q_\gamma$ and $Q_\alpha\cap Q_\gamma$ are not full-dimensional. Therefore,

$$(***):\;\;\;\; L_{\ZZ^{3n}}\Bigl(\fC_{3,n}\cap H_{\{1\},\alpha}^-\cap H_{\{1\}\beta}^-\cap H_{\{1\},\gamma}^-\Bigr)=L_{\ZZ^{3n}}(Q_0)-L_{\ZZ^{3n}}(Q_\alpha)-L_{\ZZ^{3n}}(Q_\beta)-L_{\ZZ^{3n}}(Q_\gamma).$$

We will conclude the lemma by computing the volumes of $Q_g$ for $g\in \zz$.
We start with the polytope $Q_0$. It is defined by only $3n+1$ inequalities, which means that it is a simplex, providing that it is full-dimensional. We will simply list all of its $(3n+1)$ vertices, it is straightforward to check that each of them satisfies $3n$ equalities and one inequality:

$$0,\ -e_\alpha^1+e_\beta^1+e_\gamma^1,\ e_\alpha^1-e_\beta^1+e_\gamma^1,\ e_\alpha^1+e_\beta^1-e_\gamma^1,$$

$$e_g^1+e_g^i \text{ for all } 2\le i\le n, g\in\{\alpha,\beta,\gamma\}.$$

Its lattice volume is the determinant of the matrix whose rows are all vertices except 0. It is a block lower triangular matrix, thus its determinant is equal to 
$$
\det\begin{pmatrix}
-1&1&1\\
1&-1&1\\
1&1&-1
\end{pmatrix}
=4.$$

Now we compute the volume of the polytope $Q_\gamma$. Clearly, all the polytopes $Q_\alpha$, $Q_\beta$, and $Q_\gamma$ have the same lattice volume. From the inequalities for $Q_\gamma$ we get:
\[
S_{\{1\},\gamma}(x)=\sum_{i=2}^n (x_\alpha^i+x_\beta^i)-(x_\alpha^1+x_\beta^1) \le
\]
\[
\le \sum_{i=2}^n (x_\alpha^i+x_\beta^i+2x_\gamma^i)-(x_\alpha^1+x_\beta^1+2x_\gamma^1)=S_{\{1\},\alpha}(x)+S_{\{1\},\beta}(x)\le 0.
\]

Therefore, the inequality $S_{\{1\},\gamma}(x)\le 0$ is redundant since it follows from other inequalities. This means that also polytope $Q_\gamma$ is defined by only $3n+1$ inequalities and it is a simplex. Again, we list all of its vertices:

$$0,\ e_\alpha^1,\ e_\beta^1,\ e_\alpha^1+e_\beta^1-e_\gamma^1,$$
$$e_\alpha^1+e_\alpha^i,\ e_\beta^1+e_\beta^i,\ \frac 12 (e_\alpha^1+e_\beta^1+e_\gamma^i) \text{ for all } 2\le i\le n.$$

The volume of this simplex is the absolute value of the determinant of the matrix whose rows are all vertices of $Q_\gamma$ except 0. Since it is a lower triangular matrix its determinant is $-(1/2)^{n-1}$, therefore the lattice volume of $Q_\gamma$ is $1/2^{n-1}$.

Now we simply plug the lattice volumes of $Q_g$ in $(***)$ to obtain:

$$L_{\ZZ^{3n}}\Bigl(\fC_{3,n}\cap H_{\{1\},\alpha}^-\cap H_{\{1\}\beta}^-\cap H_{\{1\},\gamma}^-\Bigr)=4-\frac{3}{2^{n-1}}.$$
\end{proof}

\begin{Theorem}
For all $n\ge 2$, the lattice volume of the polytope $P_{\zz,n}$ in the lattice $L_{\zz,n}$ is $$\frac{(3n)!}{4\cdot 6^n}-3\cdot 2^{n-3}\cdot \sum_{i=0}^n (-2)^i \binom ni \frac{(3n)!}{(2n+i)!} + 3\cdot 4^{n-2}\binom{2n}{n}-n\cdot 4^{n-1}.$$
\end{Theorem}

\begin{proof}
From the facet description of the polytope $P_{\zz,n}$, Theorem~\ref{facetdescription_zz}, we can see that $$P_{\zz,n}=\fC_{3,n}\setminus \left(\bigcup_{g\in\{\alpha,\beta,\gamma\}}\bigcup_{A\in \mathcal P_{\text{odd}}([n])} (\fC_{3,n}\cap H_{A,g}^-) \right).$$  

First, we compute $V_{\ZZ^{3n}}(\fC_{3,n})$. Since it is a product of 3-dimensional simplices, we have $V(\fC_{3,n})=(1/6)^n$ and therefore $V_{\ZZ^{3n}}(\fC_{3,n})=(3n)!/6^n$.

We will use the principle of inclusion and exclusion to compute the volume of $P_{\zz,n}$:

$$V_{\ZZ^{3n}} (P_{\zz,n})=V_{\ZZ^{3n}}(\fC_{3,n})+\sum_{k=1}^{3\cdot 2^{n-1}} (-1)^k \sum_{\substack{K\subset \mathcal{P}_{\text{odd}}([n])\times\{\alpha,\beta,\gamma\}\\ |K|=k}}
V_{\ZZ^{3n}}\left(\fC_{3,n} \cap \bigcap_{(A,g)\in K} H_{A,g}^- \right).$$

By Lemma~\ref{z22-emptyintersection}, if we have $g_1=g_2$ for two different elements $(A_1,g_1),(A_2,g_2)\in K$, then $$V_{\ZZ^{3n}}\left(\fC_{3,n} \cap \bigcap_{(A,g)\in K} H_{A,g}^- \right)=0.$$ Thus, we may sum only through such sets $K$ which does not contain such elements. In particular, we also can sum up only through terms with $k\ge 3$. We obtain the following:

$$V_{\ZZ^{3n}} (P_{\zz,n})=\frac{(3n)!}{6^n}-\sum_{(A,g)\in \mathcal{P}_{\text{odd}}([n])\times\{\alpha,\beta,\gamma\}} V_{\ZZ^{3n}}\left(\fC_{3,n} \cap H_{A,g}^- \right)+$$
$$+\sum_{A,B\in \mathcal{P}_{\text{odd}}([n])}\sum_{(g,h)\in \{(\alpha,\beta),(\beta,\gamma),(\alpha,\gamma)\}} V_{\ZZ^{3n}}\left(\fC_{3,n} \cap H_{A,g}^-\cap H_{B,h}^- \right)-$$
$$-\sum_{A,B,C \in \mathcal{P}_{\text{odd}}([n])} V_{\ZZ^{3n}}\left(\fC_{3,n} \cap H_{A,\alpha}^-\cap H_{B,\beta}^-\cap H_{C,\gamma}^- \right).$$

Now we use Lemmas \ref{z22-1facet},\ref{z22-2facet},\ref{z22-3facet} for the terms in the sums. All terms in the first two sums become equal. In the last sum, we can only sum through such triples $(A, B, C)$ for which $|\Delta(A, B, C)|=1$. The number of such triples of sets is $n\cdot |\mathcal{P}_{\text{odd}}([n])|^2=n\cdot 4^{n-1} $ because for every choice of $A, B$ there are exactly $n$ sets $C$ satisfying the condition. This happens because $A,B$ and $\Delta(A,B,C)$ uniquely determine $C$. In fact, $\Delta(A,B,\Delta(A,B,C))=C$. We obtain

$$V_{\ZZ^{3n}} (P_{\zz,n})=\frac{(3n)!}{6^n}-3\cdot 2^{n-1}\cdot \sum_{i=0}^n (-2)^i \binom ni \frac{(3n)!}{(2n+i)!} +$$
$$+3\cdot 4^{n-1}\cdot\left(\binom{2n}{n}-\frac{n}{2^{n-1}}\right)- n\cdot 4^{n-1}\cdot \left(4-\frac 3{2^{n-1}}\right)=$$
$$=\frac{(3n)!}{6^n}-3\cdot 2^{n-1}\cdot \sum_{i=0}^n (-2)^i \binom ni \frac{(3n)!}{(2n+i)!} + 3\cdot 4^{n-1}\binom{2n}{n}-n\cdot 4^{n}.$$

This is the lattice volume in the lattice $\ZZ^{3n}$. However, we are interested in the lattice volume in the lattice $L_{\zz,n}$ which is the lattice generated by the vertices of the polytope $P_{\zz,n}$. Since $L_{\zz,n}$ is a sublattice of $\ZZ^{3n}$ of index 4, we simply divide by 4 to conclude the theorem.

\end{proof}

\begin{Corollary}\label{phylodegZ2xZ2}
The phylogenetic degree of the projective algebraic variety $X_{\zz,n}$ is equal to $$\frac{(3n)!}{4\cdot 6^n}-3\cdot 2^{n-3}\cdot \sum_{i=0}^n (-2)^i \binom ni \frac{(3n)!}{(2n+i)!} + 3\cdot 4^{n-2}\binom{2n}{n}-n\cdot 4^{n-1}.$$
\end{Corollary}

\vspace{.2in}
\section{Phylogenetic degrees of \texorpdfstring{$X_{\ZZ_3, n}$}{Z3}}\label{3}
\vspace{.2in}

In this section, we consider the variety $X(\mathbb{Z}_3, T)$, where $T$ is the $n$-claw tree. In order to compute the degree of this variety, we will compute the volume of its associated polytope, $P(\ZZ_3, n)$. We proceed similarly as in the previous two sections. We start with the product of simplices $\fC_{2,n}:=(\Delta_2)^n$, where $\Delta_2=\{x\in\RR^2: x_1,x_2\ge 0,x_1+x_2\le 1\}$ is the 2-dimensional unit simplex.
Again, we cut off parts that are not in $P_{\ZZ_3,n}$ and use the principle of inclusion and exclusion.

 \begin{Theorem}[\cite{RM}, Theorem 3.2]\label{facetsZ3}
The facet description of the polytope $P(\ZZ_3, n)$ is given by: $$x_i^j\ge 0, \text{ for } 1\le i\le 2,1\le j\le n,$$
$$x_1^j+x_2^j\le 1, \text{ for } 1\le j\le n,$$
$$\langle u_{a_1},x^1\rangle+\langle u_{a_2},x^2\rangle+\dots+\langle u_{a_n},x^n\rangle\ge 2-a_1-a_2-\dots-a_n,$$
$$\langle w_{a_1},x^1\rangle+\langle w_{a_2},x^2\rangle+\dots+\langle w_{a_n},x^n\rangle\ge 2-a_1-a_2-\dots-a_n,$$
for all $(a_1,a_2,\dots,a_n)\in \{0,1,2\}^n$ such that $a_1+a_2+\dots+a_n\equiv 2\pmod 3$, where $u_0=(1,2),u_1=(1,-1),u_2=(-2,-1),w_0=(2,1),w_1=(-1,1),w_2=(-1,-2)$ and $x^j=(x_1^j,x_2^j).$
\end{Theorem}

Let us denote, for $(a_1,\dots,a_n)=:A\in\{0,1,2\}^n$,
$$S_{A,1}(x)=\sum_{i=1}^n \langle u_{a_i}, x^i\rangle,$$
$$S_{A,2}(x)=\sum_{i=1}^n \langle w_{a_i}, x^i\rangle.$$

Similarly, as in the $\ZZ_2$ and $\zz$ cases, let us denote for $A\in\{0,1,2\}^n$ and $g\in\{1,2\}$: 
$$H^+_{A,g}=\{x\in \RR^{2n}: S_{A,g}(x)\ge 2-\sum_{i=1}^n a_i\},$$
$$H^-_{A,g}=\{x\in \RR^{2n}: S_{A,g}(x)\le 2-\sum_{i=1}^n a_i\},$$
$$H_{A,g}=\{x\in \RR^{2n}: S_{A,g}(x)= 2-\sum_{i=1}^n a_i\}.$$

In this section we will be working with $n$-tuples from $\{0,1,2\}^n$, let us also denote $\mathbf{0}:=(0,0,\dots,0)\in \{0,1,2\}^n$ and $E_i:=(0,0,\dots,0,1,0,\dots,0)\in \{0,1,2\}^n $, where $1$ is on the $i$-th position.

\begin{Lemma}\label{z3-emptyintersection}
For all $g\in\{1,2\}$ and all different $n$-tuples $A,\ B\in \{0,1,2\}^n$ with $\sum_{i=1}^n {a_i}\equiv \sum_{i=1}^n b_i \pmod 3$, and $|\{i\in[n]: a_i\neq b_i\}|>2$, we have that $\fC_{2,n}\cap H_{A,g}^-\cap H_{B,g}^-$ is not full-dimensional, thus its (lattice) volume is zero.
\end{Lemma}

\begin{proof}
Without loss of generality, let $g=1$. We can consider $A, B$ as elements of $\ZZ_3^n$. This allows us to act with $A$ on $\fC_{2,n}\cap H_{A,1}^-\cap H_{B,1}^-$. Note that $A(H_{A,1}^-)=H_{\mathbf{0},1}^-,\ A(H_{B,1}^-)=H_{B-A,1}^- $, where $B-A$ is taken in $\ZZ_3^n$. Thus, we may assume $A=\mathbf{0}$.

First, suppose that in $B$ there are at least three elements equal to one. Without loss of generality, we may consider $b_1=b_2=b_3=1$. We sum up inequalities $2\ge S_{\mathbf{0},1}$ and $ 2( 2-\sum_{i=1}^n b_i)\ge 2S_{B,1}$:

$$-2\sum_{i=4}^n b_i \ge 3x_1^1+3x_1^2+3x_1^3+\sum_{i=4}^n (x_1^i+2x_2^i+2\langle u_{b_i},x^i\rangle)\ge 0+2\sum_{i=4}^n (-b_i).$$

In the last inequality, we use the fact $\langle u_{b_i},x^i \rangle \ge -{b_i}$, which can be easily verified for all values $b_i=0,1,2$. For these inequalities to hold, there must be equality everywhere. In particular $2=S_{\mathbf{0},1}$ which means that $\fC_{2,n}\cap H_{\mathbf{0},1}^-\cap H_{B,1}^-$ lie in the hyperplane $H_{\mathbf{0},1}$ and it is not full-dimensional.

If there are at least two indices $i$ such that $b_i=2$, by acting with $(2,2,\dots,2)$ we get to the previous situation.

Since $\sum_{i=1}^n b_i\equiv 0\pmod 3$, we are left only with the case that in $n$-tuple $B$ there are exactly two ones, two twos, and the rest is zero.
Without loss of generality, we may consider $ b_1=b_2=1,\ b_3=b_4=2$.

In this case, we simply sum up $-4\ge 2S_{B,1}$ and $2\ge S_{\mathbf{0},1}$ to obtain:

$$-2\ge 2x_1^1+x_2^1+2x_1^2+x_2^2-x_1^3+x_2^3-x_1^4+x_2^4+\sum_{i=5}^n (2x_1^i+4x_2^i+\ge -2.$$

Again, there must be equality in all inequalities, and in particular, $2= S_{\mathbf{0},1}$. This implies that, also in this case, $\fC_{2,n}\cap H_{\mathbf{0},1}^-\cap H_{B,1}^-$ is not full-dimensional.
\end{proof}

\begin{Lemma}\label{z3-twoofsametype}
For all $g\in\{1,2\}$ and all different $n$-tuples $A,\ B\in \{0,1,2\}^n$ with $\sum_{i=1}^n {a_i}\equiv \sum_{i=1}^n b_i \equiv 2\pmod 3$, and $|\{i\in[n]: a_i\neq b_i\}|=2$,  there exists $C\in\{0,1,2\}^n$ with $\sum_{i=1}^n\equiv 2\pmod 3 $ such that $\fC_{2,n}\cap H_{A,g}^-\cap H_{B,g}^-\subset H_{C,-g}^-$.
\end{Lemma}

\begin{proof}
Without loss of generality, we may assume $g=1$ and $\{i\in[n]: a_i\neq b_i\}=\{2,3\}$ and also $a_2-b_2\equiv b_3-a_3\equiv 1\pmod 3$.

Similarly, as in the proof of Lemma~\ref{z3-emptyintersection}, we may act with $A-E_1-E_2$, where we consider $A$ as an element of $\ZZ_3^n$. This allows us to assume, without loss of generality, also that $A=E_1+E_2$ and $B=E_1+E_3$.

By summing up the inequalities for $A$ and $B$, we obtain:

$$0\ge 2x_1^1-2x_2^1+2x_1^2+x_2^2+2x_1^3+x_2^3+\sum_{i=4}^n (2x_1^i+4x_2^i)\ge$$
$$\ge -x_1^1-2x_2^1+2x_1^2+x_2^2+2x_1^3+x_2^3+\sum_{i=4}^n (2x_1^i+x_2^i)=S_{2E_1,2}, $$

which proves the lemma.
\end{proof}

\begin{Lemma}\label{z3-twoofdiftype}
For all different $n$-tuples $A,\ B\in \{0,1,2\}^n$ with $\sum_{i=1}^n {a_i}+\sum_{i=1}^n b_i \equiv 1\pmod 3$, and $|\{i\in[n]: a_i+b_i\equiv 0\pmod 3\}|<n-1$, we have that $\fC_{2,n}\cap H_{A,1}^-\cap H_{B,2}^-$ is not full-dimensional and its (lattice) volume is zero.
\end{Lemma}

\begin{proof}
Without loss of generality, we may assume $A=\mathbf{0}$. Otherwise, we can reach this situation by acting with $A$ taken as an element of $\ZZ_3^n.$ 

We note that $x_1^i+2x_2^i+\langle w_{b_i},x^i\rangle \ge 0$ for all possible values of $b_i$. Next, we sum up inequalities for $A=\mathbf{0}$ and $B$, and use this to obtain:

$$4-\sum_{i=1}^n b_i\ge \sum_{i=1}^n (x_1^i+2x_2^i+\langle w_{b_i},x^i\rangle)\ge 0.$$

If $\sum_{i=1}^n b_i >4$, this implies that $\fC_{2,n}\cap H_{A,1}^-\cap H_{B,2}^-$ is empty. Moreover, $\sum_{i=1}^n b_i\equiv 1\pmod 3$, and $\sum_{i=1}^n b_i >1$, since there are at least two indices $i$ such that $b_i\neq 0$. This means we have to consider only the case $\sum_{i=1}^n b_i=4$. In this case, there must be equality in all inequalities above. In particular, $S_{\mathbf{0},1}=2$ which implies that $\fC_{2,n}\cap H_{\mathbf{0},1}^-\cap H_{B,2}^-$ is not full-dimensional. 
\end{proof}

\begin{Lemma}\label{z3-twotwo}
For any $n$-tuples $A,\ B,\ C,\ D\in \{0,1,2\}^n$ with $\sum_{i=1}^n {a_i}\equiv\sum_{i=1}^n b_i\equiv \sum_{i=1}^n c_i\equiv \sum_{i=1}^n d_i \equiv 2\pmod 3$ and $A\neq B,\ C\neq D$, we have that $\fC_{2,n}\cap H_{A,1}^-\cap H_{B,1}^- \cap H_{C,2}^-\cap H_{D,2}^-$ is not full-dimensional and its (lattice) volume is zero.
\end{Lemma}

\begin{proof}
We assume by contradiction that $\fC_{2,n}\cap H_{A,1}^-\cap H_{B,1}^- \cap H_{C,2}^-\cap H_{D,2}^-$ is full-dimensional. By Lemma~\ref{z3-twoofdiftype}, we know that $|\{i\in[n]: a_i+c_i\equiv 0\pmod 3\}|=n-1$. This cardinal cannot be $n$, since $\sum_{i=1}^n a_i+\sum_{i=1}^n c_i\equiv 2+2\equiv 1\pmod 3$. Without loss of generality, $A+C\equiv E_1\pmod 3$. Analogously, without loss of generality $B+C=E_2\pmod 3$. Note that we can not have $A+C\equiv B+C\pmod 3$ since this would imply $A=B$.
This implies $A-B\equiv E_1+2E_2\pmod 3$. 

Now we look at the sums $A+D$ and $B+D$ modulo 3. By the same arguments $a_i+d_i\equiv 0\pmod 3$ for all but one index $i$. The same is true for $b_i+d_i$. However, $(A+D)-(B+D)\equiv A-B\equiv E_1+2E_2\pmod 3$. There is only one way this is possible and that is $A+D\equiv E_1$ and $B+D\equiv E_2$. However, this implies $C=D$, which is a contradiction.
\end{proof}

\begin{Lemma}\label{z3-onefacet}
For all $g\in\{1,2\}$ and all $n$-tuples $A\in \{0,1,2\}^n$, the lattice volume in the lattice $\ZZ^{2n}$ of the polytope $\fC_{2,n}\cap H_{A,g}^-$ is equal to $2^n-n/2^{n-1}.$    
\end{Lemma}
\begin{proof}
Once again, we only consider the case $g=1$. Since $A(\fC_{2,n}\cap H_{A,1}^-)=\fC_{2,n}\cap H_{\mathbf{0},1}^-$, where we consider $A$ as an element of $\ZZ_3^n$, we may assume $A=\mathbf{0}.$

We define the polytopes $Q_0$ and $Q_j$, for $j\in[n]$ by the following inequalities:

\begin{itemize}
\item $x_g^i\ge 0$, for all $i\in[n]$, $g\in\{1,2\},$
\item $S_{\mathbf{0},1}(x)\le 2$,
\item $x_1^j+x_2^j\ge 1$.
\end{itemize}

For the definition of the polytope $Q_0$ we simply omit the last inequality. 

Firstly, we notice that $Q_{j_1}\cap Q_{j_2}$ is not full-dimensional for any $j_1,j_2\in[n]$, $j_1\neq j_2$. Indeed for any point $x\in Q_{j_1}\cap Q_{j_2}$ we have:

$$2\ge \sum_{i=1}^n x_1^i+2x_2^i\ge x_1^{j_1}+x_2^{j_1}+x_1^{j_2}+x_1^{j_2}\ge 2.$$

In particular, this means that $Q_{j_1}\cap Q_{j_2}\subset H_{\mathbf{0},1}$ and therefore it is not full-dimensional.

It is straightforward to check that $$\fC_{2,n}\cap H_{\mathbf{0},1}^-=Q_0\setminus\bigcup_{i=1}^n Q_i$$
which implies
$$ (\bigstar):\;\;\;\; V_{\ZZ^{2n}}(\fC_{2,n}\cap H_{\mathbf{0},1}^-)=V_{\ZZ^{2n}}(Q_0)-\sum_{i=1}^n V_{\ZZ^{2n}}(Q_i).$$

We will conclude the lemma by computing lattice volumes of polytopes $Q_j$. We start with the polytope $Q_0$ which is clearly a simplex with vertices $0,2e_1^i, e_2^i$ for $i\in[n]$. Therefore, $V_{\ZZ^{2n}}(Q_0)=2^n.$ 

We continue with the polytope $Q_1$, all other $Q_j$ are analogous. We will show that the inequality $x_1^1\ge 0$ in the facet description of $Q_1$ is redundant because it is a consequence of all other inequalities:

$$2\ge \sum_{i=1}^n(x_1^i+2x_2^i)\ge (x_1^1+2x_2^1)=2(x_1^1+x_2^1)-x_1^1\ge 2-x_1^1, $$
from which follows the desired conclusion. Thus, $Q_1$ is defined by $2n+1$ inequalities, and therefore it is a simplex (providing that it is full-dimensional). We simply list all of its $2n+1$ vertices, it is easy to check that they, indeed, are vertices of the simplex $Q_1$:

$$e_1,\ 2e_1,\ e_2,\ e_1^1+e_1^i,\ e_1^1+\frac 12 e_2^i \text{ for all } 2\le i\le n.$$

We can subtract the vertex $e_1$ from all vertices and compute the lattice volume of this simplex as a determinant of the matrix with rows:

$$e_1,\ e_2-e_1,\ e_1^i,\ \frac 12 e_2^i \text{ for all } 2\le i\le n,$$
which is, clearly, $(1/2)^{n-1}$. We now plug the volumes of $Q_i$ in $(\bigstar)$ to conclude the result:

$$V_{\ZZ^{2n}}(\fC_{2,n}\cap H_{\mathbf{0},1}^-)=2^n-\sum_{i=1}^n \frac{1}{2^{n-1}}=2^n-\frac{n}{2^{n-1}}.$$

\end{proof}

\begin{Lemma}\label{z3-twofacet}
For all different $n$-tuples $A,\ B\in \{0,1,2\}^n$ with $\sum_{i=1}^n {a_i}+\sum_{i=1}^n b_i \equiv 1\pmod 3$, and $|\{i\in[n]: a_i+b_i\equiv 0\pmod 3\}|=n-1$ the lattice volume in the lattice $\ZZ^{2n}$ of the polytope $\fC_{2,n}\cap H_{A,1}^-\cap H_{B,2}^-$ is equal to $3-1/2^{n-2}.$
\end{Lemma}
\begin{proof}
Without loss of generality, we may assume that $a_1+b_1\equiv 1\pmod 3$ and $a_i+b_i\equiv 0\pmod 3$ for $i\ge 2$. Analogously, as in the previous lemmas, we may assume $A=2E_1$ and, therefore, $B=2E_1$. Indeed, if this is not the case, we may get to this situation by acting with $A-(2,0,0,\dots,0)$, where we consider $A$ as an element of $\ZZ_3^n$.

Similarly, as in Lemma~\ref{z3-onefacet}, we introduce polytopes $Q_j$ for $j\in{0,1,2}$ by the following facet description:

\begin{itemize}
\item $x_g^i\ge 0$, for all $2\le i \le n$, $g\in\{1,2\},$
\item $x_1^1+x_2^1\le 1$
\item $S_{2E_1,1}(x)\le 0$, $S_{2E_1,2}(x)\le 0$,
\item $x_1^j\le 0$.
\end{itemize}

For the definition of $Q_0$ we simply omit the last inequality. Firstly, we notice that for any $x\in Q_0$, and, thus, also for any $x\in Q_1,Q_2$ it holds:

$$0\ge S_{2E_1,1}(x)+S_{2E_1,2}(x)=3\sum_{i=2}^n (x_1^i+x_2^i)-3(x_1^1+x_2^1)\Leftrightarrow x_1^1+x_2^1\ge \sum_{i=2}^n (x_1^i+x_2^i). $$

Since $x_1^1+x_2^1\le 1$, it follows that $x_1^i+x_2^i\le 1$ for all $2\le i\le n$. This implies that $Q_0\setminus(Q_1\cup Q_2)=\fC_{2,n}\cap H_{2E_1,1}^-\cap H_{2E_1,2}^-$. Moreover, $Q_1\cap Q_2$ is not full-dimensional (actually it is just a point), since $x_1^1,x_2^1\le 0$ yields in $x_1^1+x_2^1\le 0$ and, therefore, $x_1^1=x_2^1=0$.

Thus, $$ (\bigstar \bigstar):\;\;\;\; V_{\ZZ^{2n}}\fC_{2,n}\cap H_{2E_1,1}^-\cap H_{2E_1,2}^-=V_{\ZZ^{2n}}(Q_0)-V_{\ZZ^{2n}}(Q_1)-V_{\ZZ^{2n}}(Q_2).$$

It remains to compute the volume of the polytopes $Q_i$. We start with the polytope $Q_0$. Since it is defined by $2n+1$ inequalities, it is a simplex. Here is the list of all of its vertices:

$$0,\ 2e_1-e_2,\ 2e_2-e_1,\  e_2^1+e_1^i,\ e_1^1+e_2^i, \text{ for all } 2\le i\le n.$$

The lattice volume of $Q_0$ is equal to the determinant of the matrix whose rows are all of its vertices except 0. It is easy to check that this determinant is equal to $$\det\begin{pmatrix}
2 &-1\\
-1& 2
\end{pmatrix}=4-1=3.$$

We continue with the polytope $Q_2$, since, by symmetry, the lattice volume of $Q_1$ is the same. We will show that the inequality $S_{(2,0,0,\dots,0),1}(x)\le 0$ in the facet description of $Q_2$ is redundant and, therefore, $Q_2$ is also a simplex. Indeed, we have 

$$S_{2E_1,1}=\sum_{i=2}^n (x_1^i+2x_2^i)-2x_1^1-x_2^1\le \sum_{i=2}^n (4x_1^i+2x_2^i)-2x_1^1-4x_2^1=2S_{2E_1,2}\le 0. $$

Next, we list all of the vertices of $Q_2$, it can be easily verified that they are, in fact, vertices of $Q_2$:

$$0,\ e_1^1,\ 2e_1^1-e_2^1,\ e_1^1+\frac 12 e_1^i,\ e_1^1+e_2^i,$$

The lattice volume of $Q_2$ is the absolute value of the determinant of the matrix whose rows are all of the vertices except 0. Since it is a lower triangular matrix, its determinant is equal to $(-1)\cdot (1/2)^{n-1}$. Thus, $L_{\ZZ^{2n}}(Q_2)=1/2^{n-1}$. We conclude the result by plugging in $(\bigstar \bigstar)$:

$$V_{\ZZ^{2n}}\fC_{2,n}\cap H_{2E_1,1}^-\cap H_{2E_1,2}^-=3-\frac 1{2^{n-1}}-\frac 1{2^{n-1}}=3-\frac 1{2^{n-2}}.$$
\end{proof}

\begin{Theorem}
 For all $n\ge 2$, the lattice volume of the polytope $P_{\ZZ_3,n}$ in the lattice $L_{\ZZ_3,n}$ is equal to $$\frac{(2n)!}{3\cdot 2^n}-2^{n+1}\cdot 3^{n-2}+3^{n-1}\cdot n.$$   
\end{Theorem}
\begin{proof}
From the facet description of the polytope $P_{\ZZ_3,n}$, Theorem~\ref{facetsZ3}, we can see that
    $$P_{\ZZ_3,n}=\fC_{2,n}\setminus \left(\bigcup_{g\in\{1,2\}}\bigcup_{\substack{A\in \{0,1,2\}^n\\ \sum a_i\equiv 2\pmod 3}}\fC_{2,n}\cap H_{A,g}^-\right).$$

We start by computing the volume of $\fC_{2,n}$. Since it is a product of simplices with volume $1/2$, we have $V(\fC_{2,n})=(1/2)^{n}$. Therefore, $V_{\ZZ^{2n}}=(2n)!/2^n$.

We would like to compute the volume of $P_{\ZZ_3,n}$ by using the principle of inclusion and exclusion. However, it would get complicated since there are a lot of intersections with non-zero volume. Instead, we will use a simpler formula. We claim that 
$$ (\spadesuit):\;\;\;\;V_{\ZZ^{2n}}\left(\bigcup_{g\in\{1,2\}}\bigcup_{\substack{A\in \{0,1,2\}^n\\ \sum a_i\equiv 2\pmod 3}}\fC_{2,n}\cap H_{A,g}^-\right)=$$
$$=\sum_{g\in\{1,2\}}\sum_{\substack{A\in \{0,1,2\}^n\\ \sum a_i\equiv 2\pmod 3}}V_{\ZZ^{2n}}(\fC_{2,n}\cap H_{A,g}^-)-\sum_{\substack{A,B\in \{0,1,2\}^n\\ \sum a_i\equiv \sum b_i\equiv 2\pmod 3}}V_{\ZZ^{2n}}(\fC_{2,n}\cap H_{A,1}^-\cap H_{B,2}^-). $$

To prove this claim, we consider the set $\mathcal M=\{A\in\{0,1,2\}^n;\sum_{i=1}^n a_i\equiv 2 \pmod 3.\}\times \{1,2\}$. Let us denote for $M\subset \mathcal M$
$$R_M=\bigcup_{M\subset \mathcal M, M\neq \emptyset} \left(\fC_{2,n}\cap \left(\bigcup_{(A,g)\in M} H_{A,g}^-\right)\setminus \left(\bigcup_{(A,g)\in \mathcal M\setminus M} H_{A,g}^-\right) \right).$$

We have that 
$$\bigcup_{(A,g)\in\mathcal M} \left(\fC_{2,n}\cap H_{A,g}^-\right)=\bigcup_{M\subset \mathcal M, M\neq \emptyset} R_M$$
Since the sets $R_M$ are disjoint we have $$V_{\ZZ^{2n}}\left(\bigcup_{(A,g)\in\mathcal M} \left(\fC_{2,n}\cap H_{A,g}^-\right)\right)=\sum_{M\subset \mathcal M, M\neq \emptyset} V_{\ZZ^{2n}}(R_M)$$

Analogously,
$$V_{\ZZ^{2n}}(\fC_{2,n}\cap H_{A,g}^-)=\sum_{\substack{M\subset \mathcal M\\ (A,g)\in M}} V_{\ZZ^{2n}}(R_M),$$
$$V_{\ZZ^{2n}}(\fC_{2,n}\cap H_{A,1}^-\cap H_{B,2}^-)=\sum_{\substack{M\subset \mathcal M\\ (A,1),(B,2)\in M}} V_{\ZZ^{2n}}(R_M).$$

To prove $(\spadesuit)$ we write each summand as sum of $ V_{\ZZ^{2n}}(R_M)$ for suitable sets $M$. Then we need to check that every non-zero $V_{\ZZ^{2n}}(R_M)$ for $M\neq\emptyset$ appears at the right-hand side exactly once. We go through all possible sets $M$ to show that this is, indeed, the case:

\begin{itemize}
    \item $|M|=1$, i.e $M=\{(A,g)\}$: Clearly $R_M$ appears on the right-hand side only once - it is a subset of $\fC_{2,n}\cap H_{A,g}^-$.
    \item $|M|>1$, $M=\{(A_1,g),\dots, (A_k,g)\}$ for some $g\in\{1,2\}, k\ge 2$. By Lemmas \ref{z3-emptyintersection} and \ref{z3-twoofsametype} we see that either $(\fC_{2,n}\cap H_{A_1,g}^-\cap H_{A_2,g}^-)$ is not full-dimensional or $(\fC_{2,n}\cap H_{A_1,g}^-\cap H_{A_2,g}^-)$ is contained in some $H_{B,-g}$ for suitable $B$. In both cases it means that $V_{\ZZ^{2n}}(R_M)=0$.
    \item $|M|>1$, $M=\{(A_1,g),\dots, (A_k,g),(B,-g)\}$ for some $g\in\{1,2\}, k\ge 1$. In this case $R_M$ is a subset of $\fC_{2,n}\cap H_{A_i,g}^-$, $\fC_{2,n}\cap H_{B,-g}^-$ and $\fC_{2,n}\cap H_{A_i,g}^\cap H_{B,-g}^-$ for $1\le i \le k$. This implies that on the right-hand side term $V_{\ZZ^{2n}}(R_M)$ appears exactly $k+1-k=1$ times.
    \item $M=\{(A_1,1),\dots, (A_k,1),(B_1,2),\dots (B_l,2)\}$ for some $k,l\ge 2$. By Lemma \ref{z3-twotwo} we have $V(R_M)=0$.
\end{itemize}

Now we simply plug the volumes from Lemmas \ref{z3-twofacet} and \ref{z3-twofacet} in $(\spadesuit)$. Moreover, by Lemma \ref{z3-twoofdiftype}, in the second sum we may sum only through such $n$-tuples $A$ and $B$ for which $A+B\equiv E_j\pmod 3$ for some $j\in[n]$. This leads to:

$$V_{\ZZ^{2n}}\left(\bigcup_{g\in\{1,2\}}\bigcup_{\substack{A\in \{0,1,2\}^n\\ \sum a_i\equiv 2\pmod 3}}\fC_{2,n}\cap H_{A,g}^-\right)=$$
$$=\sum_{g\in\{1,2\}}\sum_{\substack{A\in \{0,1,2\}^n\\ \sum a_i\equiv 2\pmod 3}} \left(2^n-\frac{n}{2^{n-1}}\right)-\sum_{j=1}^n \sum_{\substack{A,B\in \{0,1,2\}^n\\ \sum a_i\equiv \sum b_i\equiv 2\pmod 3 \\
A+B\equiv E_J\pmod 3}} \left(3-\frac{1}{2^{n-2}}\right)= $$
$$=2\cdot 3^{n-1}\cdot \left(2^n-\frac{n}{2^{n-1}}\right) - n\cdot 3^{n-1}\cdot \left(3-\frac{1}{2^{n-2}}\right)=2^{n+1}\cdot 3^{n-1}-n\cdot 3^n. $$

We have that
$$V_{\ZZ^{2n}}(P_{\ZZ_3,n})=V_{\ZZ^{2n}}(\fC_{2,n})- V_{\ZZ^{2n}}\left(\bigcup_{g\in\{1,2\}}\bigcup_{\substack{A\in \{0,1,2\}^n\\ \sum a_i\equiv 2\pmod 3}}\fC_{2,n}\cap H_{A,g}^-\right)
$$ $$
=\frac{(2n)!}{2^n}-2^{n+1}\cdot 3^{n-1}+n\cdot 3^n.$$

Since we are interested in the lattice volume of the polytope $P_{\ZZ_3, n}$ in the lattice $L_{\ZZ_3,n}$, which is a sublattice of $\ZZ^{2n}$ of index 3, we simply divide by 3 to conclude the desired result:

$$V_{L_{\ZZ_3,n}}(P_{\ZZ_3,n})=\frac{(2n)!}{3\cdot 2^n}-2^{n+1}\cdot 3^{n-2}+n\cdot 3^{n-1}.$$

\end{proof}
\begin{Corollary}\label{phylodegZ3}
 The phylogenetic degree of the projective algebraic variety $X_{\ZZ_3,n}$ is equal to $$\frac{(2n)!}{3\cdot 2^n}-2^{n+1}\cdot 3^{n-2}+3^{n-1}\cdot n.$$    
\end{Corollary}

\section*{Acknowledgement} RD was supported by the Alexander von Humboldt Foundation, and a grant of the Ministry of Research, Innovation and Digitization, CNCS - UEFISCDI, project number PN-III-P1-1.1-TE-2021-1633, within PNCDI III.

MV was supported by Slovak VEGA grant 1/0152/22.

\end{document}